\documentclass{proc-l}
\usepackage{graphicx}
\usepackage[all]{xy}
\usepackage[ansinew]{inputenc}
\usepackage{amsmath}
\usepackage{amscd}
\usepackage{amssymb}
\usepackage{latexsym}
\usepackage{mathrsfs}
\usepackage{hyperref}

\newtheorem{thm}{Theorem}[section]

\theoremstyle{definition}
\newtheorem{cor}[thm]{Corollary}
\newtheorem{lem}[thm]{Lemma}

\theoremstyle{remark}
\newtheorem*{remark}{Remark}
\numberwithin{equation}{section}

\def \<{\langle}
\def \>{\rangle}

\def \((  {(\!(}
\def \)) {)\!)}

\begin{document}

\title[expansions of the real field by a closed discrete set]{Defining the set of integers in expansions of the real field by a closed discrete set}%
\author{Philipp Hieronymi}
\address{Department of Mathematics \& Statistics\\
McMaster University\\
1280 Main Street West\\
Hamilton, Ontario L8S 4K1\\
Canada }
\email{P@hieronymi.de}
\thanks{}%

\date{\today}
\maketitle

\begin{abstract} Let $D\subseteq \mathbb{R}$ be closed and discrete and $f:D^n \to \mathbb{R}$ be such that $f(D^n)$ is somewhere dense. We show that $(\mathbb{R},+,\cdot,f)$ defines $\mathbb{Z}$. As an application, we get that for every $\alpha,\beta \in \mathbb{R}_{>0}$ with $\log_{\alpha}(\beta)\notin \mathbb{Q}$, the real field expanded by the two cyclic multiplicative subgroups generated by $\alpha$ and $\beta$ defines $\mathbb{Z}$.
\end{abstract}

\begin{section}{Introduction}

Let $\overline{\mathbb{R}}=(\mathbb{R},+,\cdot)$ be the field
of real numbers. The main technical result of this paper is the following theorem.

\begin{thm}\label{mainthm} Let $D\subseteq \mathbb{R}$ be closed and discrete and $f:D^n \to \mathbb{R}$ be such that $f(D^n)$ is somewhere dense. Then $(\overline{\mathbb{R}},f)$ defines $\mathbb{Z}$.
\end{thm}
\noindent By combining Theorem \ref{mainthm} with a result of Friedman and Miller \cite{sparse}, Theorem A, we obtain the following striking dichotomy.

\begin{thm}\label{dichotomy} Let $\mathcal{R}$ be an o-minimal expansion of $\overline{\mathbb{R}}$ and let $D \subseteq \mathbb{R}$ be closed and discrete. Then either
\begin{itemize}
 \item  $(\mathcal{R},D)$ defines $\mathbb{Z}$ or
 \item  every  subset of $\mathbb{R}$ definable in $(\mathcal{R},D)$ has interior or is nowhere dense.
\end{itemize}
\end{thm}

\noindent An expansion of the real field that defines the set of integers also defines every projective subset of $\mathbb{R}$. Such a structure is as wild from a model theoretic view point as it can be. In contrast to this, every expansion of the real field whose definable sets either have interior or are nowhere dense, can be considered to be well behaved. For details, see \cite{tame}.\\

\noindent The proof of Theorem \ref{mainthm} will be given in Section 2 of this paper. In the rest of this section, several corollaries of Theorem \ref{mainthm} will be presented.

\subsection*{Two discrete multiplicative subgroups}
For any $\alpha \in \mathbb{C}^{\times}$, let
\begin{equation*}
\alpha^{\mathbb{Z}}:= \{\alpha^k : k \in \mathbb{Z}\}.
\end{equation*}
 In \cite{discrete} van den Dries established that the
structure $(\overline{\mathbb{R}},\alpha^{\mathbb{Z}})$ is model theoretically tame, when $\alpha \in \mathbb{R}^{\times}$. In his paper, he axiomatized its theory, showed that it has quantifier elimination and is decidable, if $\alpha$ is recursive. In the end, he asked whether similar results can be obtained for the structure $(\overline{\mathbb{R}},\alpha^{\mathbb{Z}},\beta^{\mathbb{Z}})$, in particular whether this structure defines $\mathbb{Z}$. This question has remained open ever since and has been reraised in literature many times (see \cite{dense}, \cite{proj}, \cite{tame}, \cite{core1}, \cite{trajector}, \cite{michael}). Using Theorem \ref{mainthm}, we answer this question.

\begin{thm}\label{twogroups} Let $\alpha,\beta \in \mathbb{R}_{>0}$ with $\log_{\alpha}(\beta)\notin \mathbb{Q}$. Then $(\overline{\mathbb{R}},\alpha^{\mathbb{Z}},\beta^{\mathbb{Z}})$ defines $\mathbb{Z}$.
\end{thm}
\noindent \begin{proof} The set $\alpha^{\mathbb{N}}\cup \beta^{\mathbb{N}}$ is closed and discrete. Moreover, it is definable in
$(\overline{\mathbb{R}},\alpha^{\mathbb{Z}},\beta^{\mathbb{Z}})$ and its set of quotients is dense in $\mathbb{R}_{>0}$.
\end{proof}

\begin{remark} Many results for related structures are known. Van den Dries and Günayd\i n proved in \cite{dense} that the expansion of the real field by the product group $\alpha^{\mathbb{Z}}\cdot \beta^{\mathbb{Z}}$ does not define the set of integers. Tychonievich showed in \cite{michael} that $(\overline{\mathbb{R}},\alpha^{\mathbb{Z}}\cdot \beta^{\mathbb{Z}})$ expanded by the restriction of the exponential function to the unit interval defines the set of integers.
\end{remark}

\subsection*{Definable subgroups} It is well known that no o-minimal expansion of $\overline{\mathbb{R}}$ defines a non-trivial proper subgroup of either $(\overline{\mathbb{R}},+)$ or $(\mathbb{R}_{>0},\cdot)$. Theorem \ref{dichotomy} and \cite{proj} 1.5 allow us to generalizes this result as follows.

\begin{thm}\label{definablesubgroups} Let $\mathcal{R}$ be an o-minimal expansion of $\overline{\mathbb{R}}$ and $D\subseteq \mathbb{R}$ closed and discrete such that $(\mathcal{R},D)$ does not define $\mathbb{Z}$. Then $(\mathbb{R},+)$ has no non-trivial proper definable subgroups, and there is an $a \in \mathbb{R}$ such that every non-trivial proper definable subgroup of $(\mathbb{R}_{>0},\cdot)$ is of the form $(a^{q})^{\mathbb{Z}}$, for some $q \in \mathbb{Q}$.
\end{thm}

\noindent Note that by \cite{discrete} $(\overline{\mathbb{R}},2^{\mathbb{Z}})$ does not define $\mathbb{Z}$. Hence Theorem \ref{definablesubgroups} is optimal. But more can be said in special cases. For $r\in \mathbb{R}$, $x
^r$ denotes the function
on $\mathbb{R}$ sending $t$ to $t^r$ for $t>0$ and to $0$ for $t \leq 0$.

\begin{cor} Let $D\subseteq \mathbb{R}$ be closed and discrete and let $r$ be an irrational number. If $(\overline{\mathbb{R}},x^r,D)$ does not define $\mathbb{Z}$, then there is no non-trivial proper subgroup of either $(\mathbb{R},+)$ or $(\mathbb{R}_{>0},\cdot)$ definable in $(\overline{\mathbb{R}},x^r,D)$.
\end{cor}

\subsection*{Cyclic subgroups of the complex numbers}
Theorem \ref{mainthm} allows us to prove the following characterization for all cyclic multiplicative subgroups of complex numbers. A structure $\mathcal{R}$ is called \emph{d-minimal}, if for every $M\equiv \mathcal{R}$, every definable subset of $M$ is a disjoint union of open intervals and finitely many discrete sets.

\begin{thm} Let $S$ be an infinite cyclic subgroup of $(\mathbb{C}^{\times},\cdot)$. Then exactly one of the following holds:
\begin{itemize}
\item $(\overline{\mathbb{R}},S)$ defines $\mathbb{Z}$,
\item $(\overline{\mathbb{R}},S)$ is $d$-minimal,
\item every open definable set in $(\overline{\mathbb{R}},S)$ is semialgebraic.
\end{itemize}
\end{thm}
\noindent \begin{proof} Let $S:=(ae^{i\varphi})^{\mathbb{Z}}\subseteq \mathbb{R}^2$, where $a\in \mathbb{R}_{>0}$ and $\varphi \in \mathbb{R}$. If $a=1$, $S$ is a finitely generated subgroup of the unit circle. Hence every open definable set in $(\overline{\mathbb{R}},S)$ is semialgebraic by \cite{ayhanme2}. Further $(\overline{\mathbb{R}},S)$ is not d-minimal, since it defines a dense and codense set.\\
\noindent If $\varphi \in 2 \pi \mathbb{Q}$, the structure $(\overline{\mathbb{R}},S)$ is interdefinable with $(\overline{\mathbb{R}},a^{\mathbb{Z}})$, and so does not define $\mathbb{Z}$. It defines an infinite discrete set and is d-minimal by \cite{tame}. Further the complement of $S$ in $\mathbb{R}^2$ is open and not semialgebraic.\\
\noindent Finally let $a\neq 1$ and $\varphi \notin 2 \pi \mathbb{Q}$. Then the function
\begin{equation*}
(a_1,a_2) \mapsto \sqrt{a_1^2 + a_2^2}
\end{equation*}
is injective on $S$ and maps $(ae^{i\varphi})^k$ to $a^k$ for every $k \in \mathbb{Z}$. Further the function
\begin{equation*}
(a_1,a_2) \mapsto \frac{a_2}{\sqrt{a_1^2 + a_2^2}}
\end{equation*}
maps $(ae^{i\varphi})^k$ to $\sin(k\varphi)$ for every $k\in \mathbb{Z}$. Hence the map $f: a^{\mathbb{Z}} \to (0,1)$
\begin{equation*}
a^k \mapsto \sin(k\varphi)
\end{equation*}
is definable in $(\overline{\mathbb{R}},S)$. Since $\varphi \notin 2\pi \mathbb{Q}$, the image of $f$ is dense in $(0,1)$. Hence $(\overline{\mathbb{R}},S)$ defines $\mathbb{Z}$ by Theorem \ref{mainthm}.
\end{proof}

\noindent Expansions of the real field by finite rank multiplicative subgroups of the unit circle have been studied by Belegradek and Zilber in \cite{unit} and Günayd\i n in \cite{ayhan}.

\subsection*{A generalization of Miller's AEG} \noindent The proof of Theorem \ref{mainthm} is based on Miller's Lemma on Asymptotic Extraction of Groups from \cite{proj}:

\begin{lem} \label{millerlemma} An expansion of $\overline{\mathbb{R}}$ defines $\mathbb{Z}$ iff it defines the range of a sequence $(a_k)_{k \in \mathbb{N}}$ of real numbers such that $\lim_{k \to \infty} (a_{k+1} -a_k) \in \mathbb{R}-\{0\}$.
\end{lem}
\noindent Theorem \ref{dichotomy} also allows us to prove the following generalization of Lemma \ref{millerlemma}.

\begin{thm}\label{genmiller} An expansion of $\overline{\mathbb{R}}$ defines $\mathbb{Z}$ iff it defines the range $S$ of a strictly increasing sequence $(a_k)_{k \in \mathbb{N}}$ of positive real numbers such that $S$ is closed and discrete and  $\lim_{k\to \infty} \frac{a_{k+1}}{a_k} =1$.
\end{thm}
\noindent \begin{proof} Let $Q$ be the set of quotients of $S$. Since $\lim_{k\to \infty} \frac{a_{k+1}}{a_k}=1$, it is easy to see that $Q$ is dense in $\mathbb{R}_{>0}$. Then Theorem \ref{dichotomy} implies that $\mathbb{Z}$ is definable.
\end{proof}

\subsection*{Acknowledgement} This research was funded by the \emph{Deutscher Akademischer Austausch Dienst}. The author thanks \emph{The Fields Institute} for hospitality and Chris Miller for help in preparing this paper.

\end{section}

\begin{section}{Proof of Theorem \ref{mainthm}}

\begin{lem}\label{onedim} Let $D\subseteq \mathbb{R}$ be  closed and discrete, and let $n\in \mathbb{N}$. Then there are $E \subseteq \mathbb{R}$ and a bijection $f : D^n \to E$ such that $f$ is definable in $(\overline{\mathbb{R}},D)$, $E$ is closed and discrete, and $|a -b| \geq 1$ for all distinct $a,b \in E$.
\end{lem}
\noindent \begin{proof} We can reduce to the case that $D$ is infinite, $D\subseteq \mathbb{R}_{>0}$ and $n\geq 1$. Since $D$ is countable, there are $a_1,...,a_n \in \mathbb{R}_{>0}$ which are linearly independent over the field $\mathbb{Q}(D)$. Define $g : D^n \to \mathbb{R}_{>0}$ by
\begin{equation*}
(d_1,...,d_n) \mapsto a_1d_1 + ... + a_n d_n.
\end{equation*}
Since $a_1,...,a_n$ are linearly independent over the field $\mathbb{Q}(D)$, the map $g$ is injective. Further for every positive real number $b$ there are only finitely many elements in $g(D^n)$ which are smaller than $b$. Hence $g(D^n)$ is discrete and closed.\\
Let $\sigma :g(D^n) \to g(D^n)$ be the successor function on the well-ordered set $(g(D^n),<)$. Further let $h: g(D^n) \to \mathbb{R}_{>0}$ be the function defined by
\begin{equation*}
x \mapsto x \cdot \max \left( \{ (\sigma(y)-y)^{-1} : y \in g(D^n), y< x\} \cup \{1\}\right).
\end{equation*}
It is easy to see that $h$ is strictly increasing and definable. Hence $h$ is injective. By construction the distance between two elements of $h(g(D^n))$ is at least $1$. So set $E:=h(g(D^n))$ and $f:= h\circ g$. \end{proof}

\noindent We fix the following notation. For two sets $A,B$, we write $A-B$ for the relative complement of $B$ in $A$. For a Lebesgue measurable set $S \subseteq \mathbb{R}$, we will write $P\left[S\right]$ for its Lebesgue measure.

\noindent \begin{proof}[Proof of Theorem \ref{mainthm}] Let $D \subseteq \mathbb{R}
$ be closed and discrete and let $f: D^n \to \mathbb{R}$ be such that $f(D^n)$ is somewhere dense. By Lemma \ref{onedim}, we can assume that $D$ is a subset of $\mathbb{R}_{\geq 1}$, $n$ equals 1 and $|d-e|\geq1$ for all distinct $d,e \in D$. After composing $f$ with an affine function and shrinking $D$ we can assume that $f(D)$ is contained in the interval $(1,2)$ and dense in it.\\
First we will construct a sequence $(d_N)_{N=1}^{\infty}$ of elements of $D$ with the following properties: for all $M,N\in\mathbb{N}_{\geq 1}$
\begin{itemize}
\item[(i)] if $M<N$, then
\begin{align*}
f(d_M)(1+\frac{d_M^{-2}}{M+\frac{1}{M}}) &<f(d_N)(1+\frac{d_N^{-2}}{N+\frac{1}{N}})\textrm{ and }\\
f(d_N)(1+\frac{d_N^{-2}}{N})&<f(d_M)(1+\frac{d_M^{-2}}{M})<2,
\end{align*}
\item[(ii)] if $d \in D$, $N>1$ and $d_1\leq d_{N-1}^7<d<d_N$, then
 \begin{equation*}
 f(d)(1+d^{-2}) < f(d_N) \textrm{ or } f(d)>f(d_N)(1+d_N^{-2}).
 \end{equation*}
\item[(iii)] $d_1 > 4$ and $d_N > \max \{4,2\cdot N,d_{N-1}^{49}\}$ for $N>1$.
\end{itemize}
We now show that such a sequence exists by induction on $N$. For $N=1$, take a $d_1 \in D$ such that $d_1>4$ and $
 f(d_1)(1+d_1^{-2}) < 2$. So conditions (i) and (iii) hold for such a $d_1$. Further condition (ii) will be satisfied trivially. Suppose we have already defined a sequence $d_1,...,d_N$ with the properties (i)-(iii). By (iii) for $d_N$ and the fact that the distance between two distinct elements of $D$ is at least $1$, it follows that
\begin{align*}
P\left[\bigcup_{d\in D, d \geq d_N^7} \left[f(d),f(d)(1+d^{-2})\right]\right]
&< \sum_{d \in D, d\geq d_N^{7}} 2  \cdot d^{-2}\\
& \leq 2 \cdot d_N^{-6}\\
&< \frac{1}{N^3+N} \cdot d_N^{-2}\\
&\leq P\left[\left(f(d_N)(1+\frac{d_N^{-2}}{N+\frac{1}{N}}),f(d_N)(1+\frac{d_N^{-2}}{N})\right)\right].
\end{align*}
In the following, let $S$ be the set
\begin{equation*}
\left(f(d_N)(1+\frac{d_N^{-2}}{N+\frac{1}{N}}),f(d_N)(1+\frac{d_N^{-2}}{N})\right) -  \bigcup_{d\in D, d \geq d_N^7} \left[f(d),f(d)(1+d^{-2})\right].
\end{equation*}
By the above calculation, the set $S$ has positive Lebesgue measure. By a result of Steinhaus from \cite{steinhaus}, we can then find elements in $S$ arbitrarily close together. Hence we can take $x_1,x_2 \in S$ with $x_2 > x_1$ so close together that the smallest $d \in D$ with
$x_1 < f(d) < x_2$ satisfies
\begin{equation*}
d > \max \{2\cdot(N+1),d_N^{49}\}.
\end{equation*}
Let $d_{N+1}$ be this smallest element of $D$ with $x_1 < f(d_{N+1})< x_2$. We now show that $d_{N+1}$ satisfies (i)-(iii). Condition (iii) directly follows from the definition of $d_{N+1}$. For (i), note that $d_{N+1}>d_N^7$ and $x_1,x_2 \in S$. Thus
\begin{align}\label{thmineq}
f(d_N)(1+\frac{d_N^{-2}}{N+\frac{1}{N}})&<x_1<f(d_{N+1}) \textrm{ and }\\
f(d_{N+1})(1+d_{N+1}^{-2})&<x_2<f(d_N)(1+\frac{d_N^{-2}}{N}). \label{thmineq2}
\end{align}
So $d_{N+1}$ satisfies (i) as well. For (ii), let $d \in D$ with $d_{N}^7 \leq d < d_{N+1}$. By minimality of $d_{N+1}$, we have that either $f(d)\leq x_1$ or $f(d)\geq x_2$. By inequality \eqref{thmineq2} we just need to consider the case that $f(d)\leq x_1$.  Since $d\geq d_N^7$ and $x_1 \in S$, we get that $f(d)(1+d^{-2})<x_1$. Hence $f(d)(1+d^{-2})<f(d_{N+1})$ by inequality \eqref{thmineq}. Thus $d_{N+1}$ satisfies (ii).\\
Let $c:= \lim_{N\to \infty}f(d_N)(1+\frac{d_N^{-2}}{N})$.
Let $\nu: D - \{ c \} \to \mathbb{R}$ be the function defined by
\begin{equation*}
\nu(x):=\frac{x^{-2}f(x)}{c-f(x)}.
\end{equation*}
Let $\varphi(x)$ be the formula
\begin{equation*}
\forall u \in D  ( f(u) < c < f(u)\cdot (1+u^{-2}))\rightarrow ( u < x^{\frac{1}{7}} \vee u > x ),
\end{equation*}
and define a subset $A$ of $D$ by
\begin{equation*}
A := \{ d \in D \ : \ f(d) < c < f(d)\cdot (1+d^{-2})\wedge d \geq d_2 \wedge \varphi(d)\}.
\end{equation*}
We now show that the following two statements hold:
\begin{itemize}
\item[(1)] for every $N \in \mathbb{N}_{>0}$, $\nu(d_N) \in (N,N+\frac{1}{N})$, and
\item[(2)] $A = \{ d_N \ : \ N \in \mathbb{N}_{>1}\}$.
\end{itemize}
We first consider (1). By (i),
\begin{equation}\label{nuineq}
f(d_N)(1+\frac{d_N^{-2}}{N+\frac{1}{N}})<c<f(d_N)(1+\frac{d_N^{-2}}{N}).
\end{equation}
After easy rearrangements, one sees that \eqref{nuineq} is equivalent to the statement $\nu(d_N) \in (N,N+\frac{1}{N})$. So (1) holds.\\
For (2), let $d \in A$. For a contradiction, suppose there is $N \in \mathbb{N}_{>1}$ such that $d_{N-1} < d < d_{N}$.  By \eqref{nuineq}, we have
\begin{equation*}
f(d_{N-1}) < c < f(d_{N-1})(1+d_{N-1}^{-2}).
\end{equation*}
Since $\varphi(d)$ holds, $d_{N-1}<d^{\frac{1}{7}}$. Hence $d_{N-1}^7<d<d_{N}$.
Thus by (ii),
\begin{equation*}
f(d)(1+d^{-2}) < f(d_{N})\textrm{ or }f(d)>f(d_N)(1+d_N^{-2}).
\end{equation*}
By the definition of $c$ and (1), we get
\begin{equation*}
f(d)(1+d^{-2})< c\textrm{ or } f(d) > c .
\end{equation*}
So the inequality $f(d)< c < f(d)(1+d^{-2})$ does \emph{not} hold. This is a contradiction to the statement $d \in A$.\\
It is only left to show that for every $N \in \mathbb{N}_{>1}$, we have $d_N \in A$. By \eqref{nuineq}, it only remains to establish $\varphi(d_N)$. Therefore let $d \in D$ with $d_N^{\frac{1}{7}}<d < d_N$. Since $d_{N-1}^{49}<d_N$, we have that  $d_{N-1}^7< d$. By the above, we get that $f(d)< c < f(d)(1+d^{-2})$ does \emph{not} hold. Hence $\varphi(d_N)$. Thus (2) holds.\\
Now it is easy to see that $(\overline{\mathbb{R}},f)$ defines $\mathbb{Z}$. By (2), the set $\{d_N: N \in \mathbb{N}_{>1}\}$ is definable and so is its image under $\nu$. By (1), we have that $\lim_{N\to \infty} (N - \nu(d_N))=0$. Thus by Lemma \ref{millerlemma},
$(\overline{\mathbb{R}},f)$ defines $\mathbb{Z}$.
\end{proof}

\end{section}


\end{document}